\documentclass[12pt,reqno]{amsart}
\usepackage[top=1in, bottom=1in, left=1in, right=1in]{geometry}
\usepackage{times, amsthm, amssymb, amsmath, amsfonts, bm, graphicx,setspace}
\usepackage{mathrsfs,wasysym}
\usepackage[pagebackref=true,pdftex]{hyperref}
\usepackage[all]{xy}
\usepackage[usenames,dvipsnames]{color}
\usepackage{latexsym}
\usepackage{amscd}
\usepackage{float}
\usepackage{tikz}
\usepackage{caption}
\usepackage{enumitem}

\numberwithin{equation}{section}

\newtheorem{theorem}{Theorem}[section]

\newtheorem{lemma}[theorem]{Lemma}
\newtheorem{corollary}[theorem]{Corollary}

\theoremstyle{definition}
\newtheorem{definition}[theorem]{Definition}
\newtheorem{example}[theorem]{Example}
\newtheorem{remark}[theorem]{Remark}

\newcommand\<{\langle}
\renewcommand\>{\rangle}
\newcommand\CC{{\mathbb C}}
\newcommand\NN{{\mathbb N}}
\newcommand\RR{{\mathbb R}}
\newcommand\ZZ{{\mathbb Z}}
\newcommand\PP{{\mathbb P}}

\newcommand\del{\partial}

\newcommand{\Aut}{\operatorname{Aut}}
\newcommand\conv{{\rm conv}}
\newcommand\vol{{\rm vol}}

\begin{document}
\mbox{}
\title{On transformations of $A$-hypergeometric functions}

\author{Jens~Forsg{\aa}rd}
\address{Department of Mathematics \\
Texas A\&M University \\ College Station, TX 77843.}
\email{jensf@math.tamu.edu}

\author{Laura Felicia Matusevich}
\address{Department of Mathematics \\
Texas A\&M University \\ College Station, TX 77843.}
\email{laura@math.tamu.edu}

\author{Aleksandra Sobieska}
\address{Department of Mathematics \\
Texas A\&M University \\ College Station, TX 77843.}
\email{ola@math.tamu.edu}

\thanks{LFM was partially supported by NSF grants DMS 1001763 and DMS 1500832.}

\begin{abstract}
We propose a systematic study of transformations of
$A$-hypergeometric functions. 
Our approach is to apply changes of variables
corresponding to automorphisms of toric rings, to
Euler-type integral representations of $A$-hypergeometric functions.
We show that all linear $A$-hypergeometric transformations arise from
symmetries of the corresponding polytope. As an application of the
techniques developed here, we show that
the Appell function $F_4$ does not admit a certain kind of Euler-type integral
representation.
\end{abstract}


\maketitle

\section{Introduction}

Hypergeometric functions are among the most extensively studied
mathematical functions. The archetypal hypergeometric function is the
\emph{Gauss hypergeometric series}:
\[
{}_2F_1(a,b;c;x) = \sum_{n=0}^\infty \frac{a(a+1)\cdots(a+n-1) \cdot
  b(b+1)\cdots(b+n-1)}{n! \cdot c(c+1)\cdots(c+n-1)} \, x^n, \quad |x|<1,
\]
where $a,b,c \in \CC$ are considered parameters, and $c \notin
\ZZ_{\leq 0}$.

Transformations are one of the characteristic phenomena exhibited by
hypergeometric functions. Among the earliest noted ones are \emph{Pfaff's
transformations}:
\[
{}_2F_1(a,b;c;x) = (1-x)^{-b}{}_2F_1\bigg(b,c-a;c;\frac{x}{x-1}\bigg) = (1-x)^{-a}{}_2F_1\bigg(a,c-b;c;\frac{x}{x-1}\bigg),
\]
which are valid when $|x|<1/2$, and which can be used to explicitly
analytically continue the Gauss hypergeometric series.  We point out
that Pfaff's transformations can be proved by a change of variables in \emph{Euler's integral}
\begin{align*}
B(b,c-b) {}_2F_1(a,b;c;x) &= \int_0^1 z^{b-1}\,(1-z)^{c-b-1}\, (1-xz)^{-a}\,dz \\
&= \int_0^\infty z^{c-b-1}\,(z+1)^{a-c}\,(z+1-x)^{-a}\, dz,
\end{align*}
valid for ${\rm Re}(c) > {\rm Re}(b) > 0$ and $|x|<1$,
and where $B$ denotes the \emph{beta function}.

On the one hand, transformation formulas are abundant in the hypergeometric literature,
see, e.g., reference works such as~\cite{AAR99, EMOT53,
  nist}. Of particular note is the work of Vid{\=u}nas, 
culminating in~\cite{Vid}, which
classifies algebraic transformations for the Gauss hypergeometric
function. On the other hand, the known hypergeometric transformations mostly
involve only a few of the most classical families of hypergeometric
functions, essentially those named after Gauss, Appell, and Lauricella.

In the late 1980s Gelfand, Graev, Kapranov, and Zelevinsky~\cite{GGZ,GKZ} introduced
a way of studying multivariate hypergeometric functions, which at once unified
and vastly generalized them. The goal of this article is to study
transformations in the more general context of the \emph{$A$-hypergeometric
  functions} introduced by Gelfand, Graev, Kapranov, and
Zelevinsky. (See Section~\ref{ssec:A-hypergeometric} for details, and
in particular, Example~\ref{ex:classical} for the relationship between the classical
and $A$-hypergeometric functions.)

Let $\{F_k(\beta; x)\}_{k\in I}$, where $I\subset \ZZ$, be a family of hypergeometric
functions in the variables $x$ depending on parameters $\beta$. 
Loosely speaking, a \emph{transformation} of hypergeometric functions
is an identity involving the functions
\[
F_k(\beta T_k; \varphi_k(x)), \quad k \in I
\]
where $T_k$ is an affine function and $\varphi_k$ is an algebraic function
for all $k\in I$.
Such an identity is not necessarily given by a closed formula, and is
allowed to involve elementary functions of $x$ and $\beta$
as coefficients.
Needless to say, the task of classifying \emph{all} hypergeometric
transformations, even for the most classical hypergeometric functions
in one or two variables, is probably out of 
reach. 

The advantage of working in the context of $A$-hypergeometric
functions, where tools from combinatorics, toric geometry, $D$-module
theory and combinatorial commutative algebra are available, is that a
systematical, unified study may be undertaken. This way, we may
attempt to
understand which kinds of transformations are valid in the most
general contexts, and which are valid only for specific families of
hypergeometric functions satisfying additional (combinatorial) properties.

Just as is the case for Pfaff's transformations,
integral expressions for hypergeometric functions provide many of the proofs of the
classical transformation formulas. Here we use Euler-type
integrals~\cite{GKZ90,BFP} to provide transformations of
$A$-hypergeometric functions. To aid our purposes,
we introduce in~\eqref{eqn:IntegralRepresentation} a more symmetric
version of these integrals, 
which has not appeared before.

We point out that in order to apply changes of variables to Euler-type
integrals so as to produce transformations, both the integrand and the
cycle of integration need to be carefully controlled. A challenge in
using the integrals of~\cite{GKZ90} is that the (compact) cycles used there are not
explicitly constructed. Even when using the explicit cycle from
\cite{Beu-alg}, it is difficult to determine which
changes of variables preserve it. On the other hand, the cycles
used in~\cite{BFP} are essentially orthants, and thus easier to
control, with the drawback, however, of requiring stronger assumptions
on the integrand in order to achieve convergence than in the case of
compact cycles.

As can be seen from the previous paragraph, we consider individual
Euler-type integrals, and study their transformations. On the other
hand, $A$-hypergeometric functions are defined as solutions of
\emph{$A$-hypergeometric systems of differential equations}. If the
parameters are sufficiently generic, it is known~\cite{GKZ90,Beu,sw-irred} that such systems
have irreducible monodromy representation. As such, if a single
$A$-hypergeometric function satisfies a given transformation, by
monodromy transformations, a basis of such solutions may be obtained
that satisfy the transformation as well, with the proviso that some
coefficients may appear in order to account for choices of branches.
This provides transformations for any $A$-hypergeometric function (not
necessarily the original transformation due to the aforementioned coefficients).

Regarding classical hypergeometric transformations, we observe
that not all of these are proved by straight change of variables in an
integral representation of the corresponding function. See, e.g.,
the quadratic and higher order transformations for the Gauss
hypergeometric function. While
these transformations are valid for their $A$-hypergeometric
counterparts, we have not been able to derive them using the
techniques developed here.

This article is organized as follows.
In Section~\ref{sec:preliminaries}, we introduce $A$-hypergeometric
functions and their Euler-type integral representations. 
In Section~\ref{sec:automorphisms}, we discuss how certain changes of
variables induce transformations of $A$-hypergeometric functions.
In Section~\ref{sec:Symmetries}, we study changes of variables and transformations induced
by symmetries of the polytope underlying a given $A$-hypergeometric
systems, and show that all linear transformations of
$A$-hypergeometric functions arise from such symmetries.
In Section~\ref{sec:ElementaryAutomorphisms} we briefly consider
transformations induced by automorphisms of the monoid ring which
defines an $A$-hypergeometric system.
In Section~\ref{sec:AppellF4} we use the techniques developed in this
article to show that the Appell function $F_4$ does not admit an
Euler-type integral representation where the cycle of integration is a
rotation of the positive orthant.

Throughout this article $\NN$ denotes the set $\{0, 1,2,\dots\}$.

\section{The $A$-hypergeometric system and Euler type integrals.}
\label{sec:preliminaries}

\subsection{The $A$-hypergeometric system}
\label{ssec:A-hypergeometric}

We set 
\begin{equation}\label{eq:A}
 A =  \begin{bmatrix}
                    \,a_1 & a_2 & \dots & a_n\,
      \end{bmatrix} \in \ZZ^{d\times n}
\end{equation}
where $a_1,\dots, a_n \in \ZZ^{d}$. We assume, firstly, that the columns of $A$
span $\ZZ^d$ as a lattice and, secondly, that the all ones vector lies in the
rowspan of $A$.
Thus, $A$ has rank $d$, and there exists a vector $\xi \in \NN^{d}$ such that
\begin{equation}
\label{eqn:xi}
\xi A = (1, \dots, 1).
\end{equation}
We denote by $z$ elements in the torus $(\CC^*)^{d}$ or 
indeterminates
in the coordinate ring of this torus.

\begin{definition}
\label{def:AHypergeometric}
Let A be 
as in \eqref{eq:A}, and let $\beta \in \CC^{d}$. The \emph{$A$-hypergeometric system with parameter
  $\beta$}, denoted $H_A(\beta)$, is the following system of
linear partial differential equations:
\begin{alignat*}{2}
\del^u F- \del^v F&= 0& \qquad & \text{for all } u,v \in \NN^n \text{ such
  that } Au=Av; \\
\sum_{j=1}^n a_{ij} x_j \del_j F & = \beta_i F& & \text{for } i=1,\dots,d.
\end{alignat*}
\end{definition}

The matrix $A$ defines a projective toric variety, namely the closure
in $\PP^{n-1}$ of the image of the map $(\CC^*)^{d}\to (\CC^*)^n$
given by $z \mapsto (z^{a_1},\dots,z^{a_n})$. We refer to
this as the \emph{toric variety underlying the hypergeometric system
$H_A(\beta)$}. The coordinate ring of this variety is the semigroup
ring $\CC[\NN A]$, where $\NN A$ is the semigroup (actually, monoid)
of nonnegative integer combinations of the columns of $A$.
 
We identify the space $\CC^A = \{ (x_1, \dots, x_n) \mid x_1,\dots,x_n \in
\CC\}$ with the family of polynomials 
\begin{equation}\label{eqn:f}
\bigg\{ f(z) = \sum_{j=1}^{n} x_j\,z^{a_{j}} \,\bigg|\, (x_1,\dots,x_n)
\in \CC^n \bigg\} 
\end{equation}
where $z=(z_1,\dots,z_d)$. Note that we deviate from the
standard approach: the Newton polytope (i.e. the convex hull of the
exponent vectors of the monomials) of a polynomial $f \in \CC^A$ is at most $(d-1)$-dimensional.
Indeed, the linear form $\xi$ encodes the quasi-homogeneity
\[
f\left(\lambda^\xi z\right) = \lambda f(z),
\]
where $\lambda^\xi z = (\lambda^{\xi_1} z_1, \dots, \lambda^{\xi_d} z_d)$.
It is common in the literature to reduce the number of variables by removing this homogeneity,
and in effect consider $f$ as a $(d-1)$-variate polynomial. 

\begin{example}
\label{ex:classical}
Classical hypergeometric functions must be suitably homogenized to fit 
in the $A$-hypergeometric setting. Let us give some information on how to realize some of the
more important classical hypergeometric functions in an
$A$-hypergeometric way. Since only special matrices $A$ are
involved, we note that the $A$-hypergeometric setting is indeed more general than the
classical setting.

It can be shown that, when $A$ and $\beta$ are given by
\[
A= \begin{bmatrix} 1 & 0 & 0 & -1 \\ 0 & 1 & 0 & \phantom{-}1 \\ 0 & 0 & 1
  & \phantom{-}1 \end{bmatrix},
\quad\text{and} \quad
\beta =\begin{bmatrix} c-1 \\ -a \\ -b \end{bmatrix},
\]
then a function $F(x)$ is a solution of the Gauss hypergeometric
equation if and only if the function $x_1^{c-1}x_2^{-a}x_3^{-b}F(x_1x_4/x_2x_3)$ is
a solution of $H_A(\beta)$.

In a similar way, the hypergeometric functions ${}_pF_{p-1}$ can be
obtained using a $(2p-1)\times 2p$ matrix $A$ consisting of a
$(2p-1)\times(2p-1)$ identity matrix to which we adjoin an additional column
whose entries are $\pm 1$ in such a way that the kernel of $A$ is
spanned by $(1,\dots,1,-1,\dots,-1)$ with $p$ ones and $p$ negative ones.

For the Appell functions, each of $F_1$ and $F_4$ arises from a single
$4\times 6$ matrix $A$ (see the proof of Theorem~\ref{thm:AppellF4}
for more details on $F_4$), while each of $F_2$ and $F_3$ correspond to a $5
\times 7$ matrix $A$. 

More generally, the $m$-variate Lauricella function $F_A$ corresponds
to a matrix $A$ of size $(2m+1)\times(3m+1)$, the function $F_B$ corresponds
to a matrix $A$ of size $(2m+1)\times(3m+1)$, the function $F_C$ corresponds to a
matrix $A$ of size $(m+2)\times (2m+2)$ (given explicitly in Example~\ref{ex:LauricellaFC}) and,
finally, the function $F_D$ corresponds to a
matrix $A$ of size $(m+2)\times (2m+2)$.
\end{example}

\subsection{Euler type integrals}

In this article, a central role is played by solutions of $H_A(\beta)$
representable by an integral
\begin{equation}
 \label{eqn:IntegralRepresentation}
 F_\sigma(\beta; x) = \int_\sigma z^{-\beta}\,f(z)^{\<\xi, \beta\>}\,d\eta,
\end{equation}
where $f$ is as in \eqref{eqn:f} and $\xi$ is as in \eqref{eqn:xi}.
The relationship between $A$-hypergeometric functions
and toric geometry is here manifest through the Haar measure
\[
d\eta = \frac{dz_1}{z_1}\wedge\dots \wedge \frac{dz_d}{z_d}
\]
of the complex torus $(\CC^*)^{d}$.
We discuss the multivaluedness of \eqref{eqn:IntegralRepresentation}
in \S~\ref{ssec:multivaluedness}.
In order for the integral~\eqref{eqn:IntegralRepresentation} to be
well defined, we need the $d$-cycle $\sigma$
to avoid the singular locus of the integrand. That is, we require that
$\sigma \subset (\CC^*)^{d}\setminus V(f)$, where $V(f)$ denotes the
zero locus of the polynomial $f$.
The inclusion $\sigma \subset (\CC^*)^{d}\setminus V(f)$ is also valid if we
perturb the coefficients of the polynomial $f$. Hence, the
integral~\eqref{eqn:IntegralRepresentation} defines a germ of a
meromorphic function 
in $x$, under the assumption that it converges.

As we do not follow the standard approach, we provide a proof of the following theorem.
\begin{theorem}[{\emph{almost} \cite[Theorem 2.7]{GKZ90}}]
\label{thm:AHypergeometric}
Provided sufficient convergence properties, the integral
\eqref{eqn:IntegralRepresentation} defines a germ of an $A$-hypergeometric
function in the variables $x$.
\end{theorem}

Let us comment on the sufficient convergence properties required in Theorem~\ref{thm:AHypergeometric}. 
In order to perform the steps in the following proof, we need to interchange the order of the integration in \eqref{eqn:IntegralRepresentation} and partial derivatives with respect to parameters $\beta$. 
This interchange is allowed if, e.g., the cycle $\sigma$ is compact as in \cite{GKZ90} or, in case of a noncompact
cycle $\sigma$, provided the integrand has sufficiently rapid decay along its boundary
as in \cite{BFP}. In the latter case one might be forced to impose restrictions on the parameter
$\beta$; however, these restrictions can be handled by considering a meromorphic extension in 
the sense of Riesz and Hadamard. 
As the explicit cycles considered in this work are of the forms appearing in \cite{GKZ90} and \cite{BFP},
we refer the reader to those articles for further details on these aspects.

\begin{proof}[Proof\/ of\/ Theorem \ref{thm:AHypergeometric}]
We need to show that $F_\sigma(\beta; x)$ solves the differential equations set forward in 
Definition~\ref{def:AHypergeometric}. We consider them in order.

Firstly, let $u \in \NN^n$. Then,
\begin{equation}
\label{eqn:AHypergeometricProof1}
\partial^u F_\sigma(\beta; x) = (\<\xi, \beta\>)_{|u|}\int_\sigma z^{-\beta}\,f(z)^{\<\xi, \beta\>-|u|}\, z^{Au}\,d\eta,
\end{equation}
where $|u| = u_1 + \dots + u_n$ and $(\<\xi, \beta\>)_{|u|}$ denotes
the descending factorial $(\alpha)_m = \prod_{j=0}^{m-1}(\alpha-j)$. 
Since $u$ has only nonnegative components, we have that $|u| = \xi Au$, and hence
the right hand side of \eqref{eqn:AHypergeometricProof1} depends, in terms of $u$, only on the vector $Au$. It follows that $F_\sigma(\beta; x)$
solves the first set of equations in Definition~\ref{def:AHypergeometric}.

Secondly, we have that
\[
\sum_{j=1}^n a_{ij} x_j \del_j F_\sigma(\beta; x)  = \int_\sigma z^{-\beta}\, z_i \,f'_i(z)\, \<\xi, \beta\>\, f(z)^{\<\xi, \beta\>-1}\, d\eta = \beta_i F_\sigma(\beta; x),
\quad i = 1, \dots, d,
\]
where the last equality is obtained through integration by parts with respect to $z_i$. This is precisely
the second set of equations in Definition~\ref{def:AHypergeometric}.
\end{proof}

\begin{remark}
\label{rem:dehomogenization}
The integral \eqref{eqn:IntegralRepresentation} is a homogeneous version of the hypergeometric
integrals appearing in \cite{GKZ90} and \cite{BFP}. In those papers, the matrix $A$ was given in the form
\begin{equation}
\label{eqn:StandardA}
A = \begin{bmatrix}
	1&\cdots&1&	0&\cdots&0&	\quad& 0&\cdots&0\\ 
	0&\cdots&0&	1&\cdots&1&	\quad& 0&\cdots&0\\
\vdots&\ddots&\vdots&	\vdots&\ddots&\vdots& \hspace{-2mm} \ddots\hspace{-2mm} 
& \vdots&\ddots&\vdots\\
	0&\cdots&0& 0&\cdots&0&\quad& 1&\cdots&1\\
	a_{11}&\cdots&a_{1n_1} & 
		a_{21}&\cdots&a_{2n_2}  & 
		\quad & a_{m_1}&\cdots&a_{mn_m}
\end{bmatrix},
\end{equation}
where $n = n_1 + \dots + n_m$ and $a_{ij} \in \ZZ^{r}$. Thus, $d = m + r$. The existence of $\xi$
ensures that $A$ can always be written in this form with $m \geq 1$. Notice that,
with $A$ as in \eqref{eqn:StandardA}, 
\[
\xi = (\underbrace{1,  \dots, 1}_{m \text{ times}}, \underbrace{0,  \dots, 0}_{r \text{ times}}).
\]
It holds that $\CC^A \simeq \prod \CC^{A_i}$, where $i = 1, \dots, m$, and where $\CC^{A_i}$ denotes the family of $r$-variate polynomials
\[
\bigg\{ f_i(w) = \sum_{j=1}^{n_i} x_{ij}\,w^{a_{ij}} \,\bigg|\, (x_{i1},\dots,x_{in_i})
\in \CC^{n_i} \bigg\},
\]
where $w = (w_1, \dots, w_r)$. 
The Euler-type hypergeometric integral considered in, e.g., \cite{GKZ90} and \cite{BFP} 
takes the form
\[
M_\varepsilon(\beta; x) = \int_\varepsilon w_1^{-\beta_{m+1}}\cdots w_r^{-\beta_{m+r}}\prod_{i=1}^m f_i(w)^{\beta_i}\, d\eta.
\]
where $\varepsilon$ is some $r$-cycle. 
Let us consider instead the integral from Definition~\ref{eqn:IntegralRepresentation},
and the change of variables $z \mapsto \varphi(z)$ defined by 
\begin{align*}
z_i &\mapsto z_i / f_{i}(w), \quad \text{for } i = 1, \dots, m, \text{ and}\\
z_i &\mapsto w_{i-m}, \quad \text{for } i = m+1, \dots, d.
\end{align*}
If $\sigma = \varphi^{-1}(\tau) \times \varepsilon$, where $\tau$ is some $m$-cycle, then
\[
F_\sigma(\beta; x) = K(\beta; \tau)\cdot M_\varepsilon(\beta; x)
\]
where
\begin{equation}
\label{eqn:DehomogenizedConstant}
K(\beta; \tau) = \int_\tau z_1^{-\beta_1}\cdots z_{m}^{-\beta_{m}}\,(z_1 + \dots + z_{m})^{\beta_1 + \dots + \beta_{m}} \, d \eta
\end{equation}
is constant with respect to $x$. We say that $F_\sigma(\beta; x)$ is the homogenized version of $M_\varepsilon(\beta; x)$
or, conversely, that $M_\varepsilon(\beta; x)$ is a dehomogenized version of $F_\sigma(\beta; x)$. In this nomenclature, the
dependence on the cycle $\tau$ is suppressed; typically one considers $\tau$ as simple as possible to ensure
convergence and nontriviality of $K(\beta; \tau)$, e.g., as the skeleton of a polydisc centered at the origin.
Throughout this text, we consider dehomogenized integrals in the examples, to emphasize
the relationship with integral representations of classical hypergeometric functions
\end{remark}

\begin{remark}
A parameter $\beta$ is called \emph{nonresonant} (with respect to $A$) if it does not
belong to any integer translate of a supporting hyperplane of a facet
of the real cone spanned by the columns of the matrix $A$. Thus,
nonresonant parameters are \emph{very generic} in the sense that they
avoid a locally finite (but infinite) collection of algebraic varieties.
It follows from Remark~\ref{rem:dehomogenization} and~\cite{GKZ90} that if $\beta$ is nonresonant,
one can always
find a basis of the solution space of $H_A(\beta)$ by considering integrals \eqref{eqn:IntegralRepresentation}
taken over $\vol(A)$-many distinct cycles $\sigma$.
\end{remark}

\subsection{Multivaluedness}
\label{ssec:multivaluedness}

An additional difficulty to overcome is the multivaluedness of the integrand in \eqref{eqn:IntegralRepresentation}.
Indeed, we need to choose branches of the exponential functions $f(z)^{\<\xi, \beta\>}$ and
$z^{-\beta}$. A change of branch alters the value of the integral by a
factor of $e^{\<\gamma, \beta\>}$, where $\gamma\in \Gamma$ for some
subgroup $\Gamma  = \Gamma(\sigma) \subset \CC^{d}$ 
depending on the cycle $\sigma$. 
(Note that the one-dimensional subspace of the solution space generated by
\eqref{eqn:IntegralRepresentation} is well-defined.)
In particular, the validity of any identity involving Euler-type integrals
depends on an appropriate choice of branches.
We mention this, as it has implications relevant for any study comparing transformations
with computations of monodromy; the coefficient $e^{\<\gamma, \beta\>}$ can appear
through the action of analytic continuation.

\begin{example}
\label{ex:multivaluedness}
Consider the $A$-hypergeometric system defined by the matrix
\[
A=\begin{bmatrix} 1 & 1 & 1\\ 0 & 1 & 2\end{bmatrix}.
\]
Let us consider a neighbourhood of a point $(x_1, x_2, x_3) \in \RR^3$ such that $x_2^2 < 4 x_1 x_3$,
and assume that $\beta = (\beta_1, \beta_2)$ is sufficiently generic
(nonresonant suffices).
Then, the solution space of the system $H_A(\beta)$ is spanned by the two dehomogenized 
Euler-type integrals
(as in \cite{BFP})
\begin{align*}
F_1(\beta_1, \beta_2; x_1, x_2, x_3) &= \int_{\RR_+} (x_1 + x_2 z + x_3 z^2)^{\beta_1}\, z^{-\beta_2}\, d\eta\\
\text{and}\quad F_2(\beta_1, \beta_2; x_1, x_2, x_3) & = \int_{\RR_-}(x_1 + x_2 z + x_3 z^2)^{\beta_1}\, z^{-\beta_2}\, d\eta.
\end{align*}
Here $\RR_+$ and $\RR_-$ refer to the nonnegative and nonpositive real
half-lines, respectively.
By the multivaluedness of the exponential functions, 
the integrals $F_1$ and $F_2$ are well-defined only up to
a multiple of $e^{2\pi i \<k, \beta\>}$ for $k \in \ZZ^2$.
Applying the change of variables $z \mapsto 1/z$ for each of the above integrals we obtain
the transformation
\begin{align*}
F_j(\beta_1, \beta_2; x_1, x_2, x_3)  
& = F_j(\beta_1, 2\beta_1 - \beta_2; x_3, x_2, x_1), \quad \text{for } j= 1, 2.
\intertext{On the other hand the change of variables $z \mapsto e^{i\pi}z$ gives}
F_1(\beta_1, \beta_2; x_1, x_2, x_3) 
& = e^{-i\pi \beta_2}F_2(\beta_1, \beta_2; x_1, -x_2, x_3).
\intertext{Thus, applying the second identity twice, and the tranformation once, we obtain,}
F_1(\beta_1, \beta_2; x_1, x_2, x_3) 
& = e^{-i2\pi (\beta_1 + \beta_2)}F_1(\beta_1, 2\beta_1 - \beta_2; x_3, x_2, x_1).
\end{align*}
This seeming contradiction is caused by inconsistencies in the choices of branches.
The transformation is valid -- but it does not commute with meromorphic continuations.
\end{example}

\section{Changes of variables, Automorphisms, and Transformations}
\label{sec:automorphisms}

We saw in Example~\ref{ex:multivaluedness} that a change of variables
in the integral \eqref{eqn:IntegralRepresentation} can induce a transformation
of $A$-hypergeometric functions. We now consider the general situation. 
In order to deduce an identity, we need to perform a change of variables
under which the cycle $\sigma$ is invariant up to homotopy.
To simplify notation, we introduce the following assumptions.
\begin{center}
\emph{Let $\varphi\colon \sigma\rightarrow \sigma$ be a diffeomorphism such that
the toric Jacobian $J_\varphi(z)$ is nonvanishing.}
\end{center}
Let $g$ denote the pullback $\varphi^*(f)$. Then, performing the change of variables $z \mapsto \varphi(z)$
in \ref{eqn:IntegralRepresentation} we obtain
\begin{equation}
\label{eqn:ChangeOfVariables}
\int_\sigma z^{-\beta}\, f(z)^{\<\xi, \beta\>} \, d\eta
= \int_\sigma \varphi(z)^{-\beta}\, g(z)^{\<\xi, \beta\>}\, J_\varphi(z) \,d\eta.
\end{equation}
We now wish to interpret the right hand side of \eqref{eqn:ChangeOfVariables}
as an $A$-hypergeometric function. We present here only a sufficient result
in this direction. There is a flexibility in considering, for example,
special values of the parameters $\beta$, which could greatly
simplify the right hand side of \eqref{eqn:ChangeOfVariables}.

One natural requirement to enable the right hand side of \eqref{eqn:ChangeOfVariables}
to represent an $A$-hypergeometric function is that the pullback $\varphi^*$ defines
a map
\begin{equation}
\label{eqn:PullbackMap}
\varphi^*\colon \CC^A \rightarrow \CC^A.
\end{equation}
The form $\xi$ from \eqref{eqn:xi} induces a grading on the semigroup ring $\CC[\NN A]$.
We say that an automorphism of $\CC[\NN A]$ is \emph{homogeneous}
if it preserves homogeneous elements under this grading.
Then $\varphi$ is induced by a homogeneous automorphism of the semigroup ring $\CC[\NN A]$
if and only if~\eqref{eqn:PullbackMap} holds and $\varphi^*$ is a $\CC$-linear map. 
Conversely, we have the main theorem of this section.

\begin{theorem}
\label{thm:AutomorphismTransformations}
Assume that $\varphi$ is induced by a homogeneous automorphism of\/ $\CC[\NN A]$. 
Then, provided sufficient convergence,
the identity \eqref{eqn:ChangeOfVariables} encodes a transformation
of $A$-hypergeometric functions.
\end{theorem}

The requirement of sufficient convergence depends heavily on the automorphism and cycle 
in question. We refer to reader to \S\ref{sec:Symmetries}--\ref{sec:ElementaryAutomorphisms} for details. 
In this section we provide only the formal computations.

Before we prove Theorem \ref{thm:AutomorphismTransformations} we need a few auxiliary results.
We remark already at this point that by abuse of notation we identify 
$\varphi \in \Aut(\CC[\NN A])$ with its induced rational map
$\varphi\colon \CC^d \rightarrow \CC^d$,
whose existence follows from Lemma~\ref{lem:autFromRational}.

\begin{lemma}
\label{lem:linearPolynomials}
Let $\varphi$ be a homogeneous automorphism of\/ $\CC[\NN A]$. Then, 
there exist linear homogeneous polynomials
$p_j \in \CC[w_1,\dots,w_n]$, for $j= 1, \dots, n$, such that
$\varphi(z^{a_j}) = p_j(z^{a_1}, \dots, z^{a_n})$.
\end{lemma}

\begin{proof}
Since $\varphi$ is an automorphism, it is injective, and hence $\varphi(z^{a_j})$ is nonconstant.
Then, since $\varphi$ is a homogeneous automorphism, $p_j$ must be a polynomial with vanishing constant term for each $j= 1, \dots, n$.

Let $q_m$, for $m= 1, \dots, n$, denote the corresponding polynomials
associated to the inverse $\varphi^{-1}$ of $\varphi$. 
Let $h(z) =  p_j(z^{a_1}, \dots, z^{a_n})$. Then, since $\<\xi, a_k\> = 1$ for all $k = 1, \dots, n$,
the degree of $p_j$ is equal to the degree of $h(\lambda^\xi z)$ in $\lambda$.
We have that
\[
 z^{a_j} = \varphi^{-1}(\varphi(z^{a_j})) = \varphi^{-1}(p_j(z^{a_1}, \dots, z^{a_n})).
\]
After the substitution $z \mapsto \lambda^\xi z$ we obtain in the left hand side a
polynomial of degree one in $\lambda$.
Since the polynomials $q_m$ all have positive degree, we obtain in the 
right hand side a polynomial in $\lambda$ of degree at least the
degree of $p_j$. Hence, the degree of $p_j$ is at most one, which implies that it
is equal to one.
\end{proof}

\begin{lemma}
\label{lem:autFromRational}
Every automorphism of\/ $\CC[\NN A]$ is induced from an automorphism of
the field of rational functions $\CC(z)$.
\end{lemma}

\begin{proof}
Since the columns of $A$ span $\ZZ^{d}$ as a lattice, each of the
standard basis vectors $e_1, \dots, e_d$ has
a representation
\begin{equation}
\label{eqn:integerRepresentation}
e_i = \sum_{j=1}^n\,m_{ij}\,a_j, \quad i = 1, \dots, d,
\end{equation}
where $m_{ij} \in \ZZ$ for all $i= 1, \dots, d$ and $j = 1, \dots, n$.
Then, $\varphi$ extends to  
a map $\CC(z) \to \CC(z)$
by 
\begin{equation}
\label{eqn:autoExtension}
\varphi(z_i) = \prod_{j=1}^n\left(p_j(z^{a_1},\dots,z^{a_n})\right)^{m_{ij}}.
\end{equation}
To see that this is an automorphism of $\CC(z)$, apply the same
procedure to the inverse automorphism $\varphi^{-1}\colon\CC[\NN A]\to \CC[\NN A]$. 
\end{proof}

We remark that the converse of Lemma~\ref{lem:autFromRational} is not
true, as most automorphisms of $\CC(z)$, or of
$\CC[z^{\pm 1}]$, do not restrict to well defined maps $\CC[\NN
A]\to \CC[\NN A]$. 
For more information on automorphisms of $\CC[\NN A]$ we refer
to~\cite{BG99}, where a classification of homogeneous automorphisms of $\CC[\NN
A]$ is provided under the assumption that $\CC[\NN A]$ is normal.

\begin{proof}[Proof\/ of\/ Theorem \ref{thm:AutomorphismTransformations}]
Since $p_j$ for $j = 1, \dots, n$ are homogeneous linear forms the pullback
$\varphi^*$ defines a linear map $\varphi^*\colon \CC^A \rightarrow \CC^A$.
In particular, $g = \varphi^*(f) \in \CC^A$.

From \eqref{eqn:autoExtension} we can conclude that each $\varphi(z_i)$ is a
rational function of $z$ for $i = 1, \dots, d$. It follows that also the toric Jacobian $J_\varphi(z)$ 
is a rational function.
By use of the generalized binomial theorem,
under proper assumptions to ensure convergence,
we can expand the factor $\varphi(z)^{-\beta}J_\varphi(z)$ of the
integrand as a generalized Laurent series.
Each term of such a series is an $A$-hypergeometric function,
which finishes the proof.
\end{proof}

We have assumed that $\varphi$ restricts 
to a diffeomorphism of the cycle $\sigma$. Given an explicit automorphism, to
deduce a valid transformation, 
we must describe explicitly the cycle $\sigma$.
What sparked our investigation of the subject matter was not 
the realization provided by
Theorem \ref{thm:AutomorphismTransformations},
but the study of Euler-type $A$-hypergeometric integrals over explicit
cycles in \cite{NP} and \cite{BFP}.

\section{Linear algebraic transformations from polytope symmetries}
\label{sec:Symmetries}
Consider a matrix $A$ as in~\eqref{eq:A}, and a monomial
homogeneous automorphism $\varphi\colon \CC[\NN A] \rightarrow \CC[\NN A]$. 
That is, all polynomials $p_j$, for $j= 1, \dots, n$, are monomials. 
It follows from these assumptions and Lemma~\ref{lem:autFromRational} that
there exist vectors $t_i \in \ZZ^{d}$ such that
\[
 \varphi(z_i) = z^{t_i}, \text{ for } i = 1, \dots, d.
\]
Let us denote by $T$ the matrix $\left[t_1\,t_2\, \dots \,t_d\right]$.

The fact that $\varphi$ is a monomial homogeneous automorphism,
together with Lemma~\ref{lem:linearPolynomials},
implies that $\varphi(z^{a_j}) = z^{a_{\pi(j)}}$ for each $j=1, \dots, n$, where $\pi(j)\in \{1, \dots, n\}$.
By injectivity of $\varphi$, we conclude that $\pi \in \mathfrak{S}_n$ is a permutation of the 
columns of $A$. Let $P = P(\pi)$ denote the corresponding permutation matrix. Then,
\begin{equation}
\label{eqn:TAP}
TA = AP.
\end{equation}

The pair $(T, P)$ encodes a polytope symmetry of the Newton polytope
$\Delta_A = \conv(A)$, the convex hull of the columns of the matrix $A$. In general,
however, a polytope symmetry of $\Delta_A$ need not induce an automorphism of $\CC[\NN A]$.
If $A$ is \emph{saturated}, that is, if $\NN A = \RR_{+}A \cap \ZZ^{d}$, or
equivalently, if $\CC[\NN A]$ is normal, then any polytope symmetry does induce an element 
of $\Aut(\CC[\NN A])$.

\begin{corollary}
\label{cor:LinearTransformation}
 The monomial homogeneous automorphism $\varphi$ induces the transformation
 \begin{equation}
 \label{eqn:LinearTransformation}
  F_{\sigma}(\beta; x) = |T|\,F_\sigma(T\beta; xP),
 \end{equation}
 where $\sigma \simeq (\RR_+)^{d-1}\times S^1$ and $xP$ and $T\beta $ denotes the standard matrix multiplication.
\end{corollary}

\begin{proof}
Let us first remark on the choice of cycle. Applying a linear transformation
we can write $A$ in the form \eqref{eqn:StandardA} with $m=1$. 
In the notation of Remark \ref{rem:dehomogenization}, we set
$\tau \simeq S^1$ to ensure convergence, and nonvanishing, of the constant \eqref{eqn:DehomogenizedConstant}.
The dehomogenized integral is taken over the cycle $\varepsilon = \RR^{d-1}$. 
Note that $\varphi$ restricts to monomial transformation in $d-1$ variables
of the dehomogenized integral, which we also denote by $\varphi$.

It is clear that $\varphi$ maps $(\RR_+)^{d-1}$ into itself. As $\varphi^{-1}$
is also a monomial automorphism, we find that $\varphi$ preserves
the positive orthant. The toric Jacobian $J_\varphi(z)$ is equal to the determinant $|T|$, 
which is nonvanishing since $\varphi$ is surjective. 
Furthermore, the identity $TA = AP$ implies that,
with the notation of Theorem~\ref{thm:AutomorphismTransformations},
\[
 g(z) = \sum_{j=1}^n(xP)_j \,z^{a_j}.
\]
Finally, that $\varphi(z)^\beta = z^{T\beta}$ is immediate.
Thus, the statement follows from Theorem \ref{thm:AutomorphismTransformations}.
\end{proof}

\begin{example}
\label{ex:2F1first}
Consider the point configuration
\[
A =
\begin{bmatrix}
1 & 1 & 1 & 1 \\
0 & 0 & 1 & 1 \\
0 & 1 & 0 & 1 
\end{bmatrix}.
\]
There is a group of eight transformations generated by the polytope symmetries
encoded by the pairs
 \[
T_1 = \left[\begin{smallmatrix}
 1	& 0 	& 0\\
 0 	& 0 	& 1 \\
 0 	& 1 	& 0 
\end{smallmatrix}\right], \quad
P_1 = \left[\begin{smallmatrix}
1 	& 0 	& 0 	& 0 \\ 
0	& 0 	& 1 	& 0 \\ 
0 	& 1 	& 0 	& 0 \\ 
0	& 0 	& 0	& 1 
\end{smallmatrix}\right]
\quad \text{and}\quad
T_2 = \left[\begin{smallmatrix}
 1	& \phantom{-}0 	& 0\\
 1 	&    -1 	& 0 \\
 0 	& \phantom{-}0 	& 1 
\end{smallmatrix}\right], \quad
P_2 = \left[\begin{smallmatrix}
0 	& 0 	& 1 	& 0 \\ 
0	& 0 	& 0 	& 1 \\ 
1 	& 0 	& 0 	& 0 \\ 
0	& 1 	& 0	& 0 
\end{smallmatrix}\right].
\]
The corresponding identities for the hypergeometric functions reads as
\begin{align*}
F_\sigma( \beta_1, \beta_2, \beta_3; x_1, x_2, x_3, x_4) & = F_\sigma(\beta_1, \beta_3, \beta_2; x_1, x_3, x_2, x_4)\\
F_\sigma(\beta_1, \beta_2, \beta_3; x_1, x_2, x_3, x_4) & = F_\sigma(\beta_1, \beta_1-\beta_2, \beta_3; x_3, x_4, x_1, x_2).
\end{align*}
Using a classical integral representation of Gauss hypergeometric function ${}_2F_1(a,b; c; x)$,
with $a =-\beta_2, b =-\beta_3, c =-\beta_1 ,$ and $x = 1-\frac{x_1x_4}{x_2x_3}$,
the first transformation translates to the 
in terms of series trivial identity
\begin{align*}
{}_2F_1\left(a,b;c; x\right) & = {}_2F_1\left(b,a;c; x\right),\\
\intertext{while the second transformation translates to the Pfaff transformation}
{}_2F_1\left(a,b;c; x\right) &= (1-x)^{-b} {}_2F_1\left(c-a, b; c; \frac{x}{x-1}\right).
\end{align*}
\end{example}

\begin{example}
\label{ex:LauricellaFC}
Suitably homogenized, the Lauricella hypergeometric function $F_C^{(m)}$ is a solution to the $A$-hypergeometric
system defined by the $(m+2)\times(2m+2)$-matrix
\[
A = \begin{bmatrix}
1 & \bf{1} & 1 &  \bf{1} \\
1 & \bf{1} & 0 &  \bf{0} \\
\bf{0} &  I_m & \bf{0} & -I_m\end{bmatrix},
\]
 where $I_m$ denotes the $m\times m$ identity matrix, and bold numbers are to be 
 interpreted as a vectors of appropriate dimensions. We can view 
 $\conv(A)$ as the convex hull of two $m$-simplices in $\RR^{m+2}$,
and find two families of transformations of $A$. The first corresponds
to a permutation of the vertices of 
the simplices: let $P$ be a permutation matrix of size $m\times m$, then
 \[
\left[\begin{smallmatrix}
 1		& 0 		& \bf{0}\\
 0 		& 1 		& \bf{0} \\
 \bf{0} 	& \bf{0} 	& P 
\end{smallmatrix}\right]
A
\left[\begin{smallmatrix}
1 		& \bf{0} 	& 0 		& \bf{0} \\ 
\bf{0} 	& P^{-1} 	& \bf{0} 	& \bf{0} \\ 
0 		& \bf{0} 	& 1 		& \bf{0} \\ 
\bf{0} 	& \bf{0} 	& \bf{0}	& P^{-1} 
\end{smallmatrix}\right]
= 
\left[\begin{smallmatrix}
1 	& \bf{1} 	& 1 		& \phantom{-} \bf{1} \\
1 	& \bf{1} 	& 0 		&  \phantom{-} \bf{0} \\
\bf{0}	&  P 		& \bf{0} 	& -P
\end{smallmatrix}\right]
\left[\begin{smallmatrix}
1 	& \bf{0} 	& 0 		& \bf{0} \\ 
\bf{0}	& P^{-1} 	& \bf{0} 	& \bf{0} \\ 
0 	& \bf{0} 	& 1 		& \bf{0} \\ 
\bf{0}	& \bf{0} 	& \bf{0} 	& P^{-1} 
\end{smallmatrix}\right]
= A.
 \]
The second family of transformations corresponds to swapping two vertices between the simplices:
 \begin{align*}
 \left[\begin{smallmatrix} 
1 	 & 0 		& \bf{0} 	& \phantom{-} 0         & \bf{0} \\ 
0 	 & 1 		& \bf{0}     	& -1 	        	& \bf{0} \\
\bf{0} & \bf{0} 	& I_{k-1} 	& \phantom{-}\bf{0}  	& \bf{0} \\ 
0 	 & 0 		& \bf{0} 	& -1 	        	& \bf{0}\\ 
\bf{0} & \bf{0} 	& \bf{0}  	& \phantom{-}\bf{0}	& I_{m-k}
\end{smallmatrix}\right]
A &  =  
\left[\begin{smallmatrix}
1 	 & \bf{1} 	& \phantom{-} 1 		& \bf{1} 	& 1 		& \bf{1} 	& 1 		& \bf{1}\\ 
1 	 & \bf{1} 	& \phantom{-} 0 		& \bf{1} 	& 0 		& \bf{0} 	& 1 		& \bf{0}\\ 
\bf{0} & I_{k-1} 	& \phantom{-}\bf{0}     	& \bf{0} 	& \bf{0} 	& I_{k-1} 	& \bf{0} 	& \bf{0}\\
0 	 & \bf{0} 	& -1 	                	& \bf{0} 	& 0 		& \bf{0} 	& 1 		& \bf{0}\\
\bf{0} & \bf{0} 	& \phantom{-}\bf{0} 	        & I_{m-k} 	& \bf{0} 	& \bf{0} 	& \bf{0} 	&-I_{m-k}
\end{smallmatrix}\right],
\end{align*}
which corresponds to transposing columns $k+1$ and $m+k+2$ of $A$. 
The two families of permutations commute, generating a group of $2^m\,m!$ linear transformations of $F^{(m)}_C$.
\end{example}

Let us end this section with a partial converse of Corollary \ref{cor:LinearTransformation}.
\begin{theorem}
Let $F(\beta; x)$ be an $A$-hypergeometric function for which there exists a transformation
valid for generic parameters $\beta$;
\begin{equation}
\label{eqn:LinearConverse}
  F(\beta; x) = K(\beta)\,F(T \beta ; xP),
\end{equation}
 where $P$ is a permutation matrix and $K(\beta)$ is a constant with respect to $x$. Then, $TA = AP$. That is, $P$ encodes a polytope symmetry
 of $A$.
\end{theorem}

\begin{proof}
We recall from~\cite[Theorem 2.7 and Corollary~2.8]{Mat09} the fact that if $\beta$ is
sufficiently generic, and $F=F(x)$ is a solution of $H_A(\beta)$, then any
differential operator annihilating $F$ must belong to
$H_A(\beta)$. Using this result, the
equation~\eqref{eqn:LinearConverse} implies that $H_A(\beta) =
H_{AP}(T\beta)$.

Since the toric ideals underlying $H_A(\beta)$ and $H_{AP}(T\beta)$
coincide, we see that $A$ and $AP$ have the same integer kernel, and
therefore (since $P$ has full rank), the same rational
rowspan. 
Now considering the second set of equations from
Definition~\ref{def:AHypergeometric} for $H_A(\beta)$ and
$H_{AP}(T\beta)$, we conclude that $TA = AP$. 
\end{proof}

\section{Linear Algebraic Transformations from Elementary Automorphisms}
\label{sec:ElementaryAutomorphisms}

It follows from \cite{BG99} that in the case when $A$ is saturated all homogeneous 
automorphisms of $\NN A$ are given by the polytope symmetries considered in \S \ref{sec:Symmetries} and elementary
(toric) automorphisms which we consider in this section. These are generated by mappings
\[
z_i \mapsto z_i + t\,z^{a},
\]
where $t$ is a scalar and $a\in A$. 
These automorphisms also generate identities of hypergeometric functions. However,
in contrast to the situation in \S\ref{sec:Symmetries},
one must perform an expansion using the generalized binomial theorem,
as in the proof of Theorem \ref{thm:AutomorphismTransformations}.
This requires a specialization of $\beta$ to a family of hyperplanes in the parameter space.
To simplify the exposition, we deduce the corresponding
transformations in examples only.

\begin{example}
Consider the matrix
\[
A = \begin{bmatrix}
            1 & 1 & 1 \\
            0 & 1 & 2
      \end{bmatrix},
\]
which we consider to be in the form~\eqref{eqn:StandardA} with $m=1$.
Consider the automorphism of $\CC[\NN A]$ 
defined by $z_2 z_1 \mapsto z_2z_1 + z_1$, which
for the dehomogenized integral induces the change of variables $z \mapsto z + 1$. 
A cycle preserved under this transformation is $\sigma = \RR$,
and under the restriction that $f(z)$ is nonvanishing on $x$, the
integral~\eqref{eqn:IntegralRepresentation} converges for $\beta$ in
some open 
domain \cite{BFP}.  
We find that
\[
\int_\RR (x_1 + x_2 z + x_3 z^2)^{\beta_1} \,z^{-\beta_2}\,d\eta = \int_\RR\left((x_1 + x_2 + x_3) + (x_2 + 2 x_3) z +x_3 z^2\right)^{\beta_1}\,(z + 1)^{-\beta_2} \,d\eta.
\]
In order to obtain an identity of $A$-hypergeometric functions with shifted parameters, we need to expand
the binomial $(z+1)^{-\beta_2}$. 
Since the cycle is $\sigma = \RR$, this imposes the requirement that $\beta_2$ is integer.
For $\beta_2 = -N$, we obtain the identity
\[
F_\RR\left(x_1,x_2,x_3, \beta_1, -N\right) = \sum_{K=0}^{N-1}\,{N-1 \choose K}\,F_\RR\left(x_1+x_2+x_3, x_2+2x_3, x_3, \beta_1, 1-K\right).
\]
\end{example}

\begin{example}
Consider the matrix from Example \ref{ex:2F1first}. Then, all automorphisms
of $\CC[\NN A]$ are generated by the monomial automorphism of that example,
and the toric automorphism induced by $z_1 \mapsto z_1 + 1$
in the dehomogenized integral.
The latter gives the identity
\begin{align*}
& \int_\RR (x_1 + x_2 z_1 + x_3z_2 + x_4 z_1 z_2)^{\beta_1} \,z_1^{-\beta_2}z_2^{-\beta_3}\,d\eta\\
& \qquad\qquad= \int_\RR\left((x_1 + x_2) + x_2z_1 + (x_3 + x_4) z_2 +x_4z_1 z_2\right)^{\beta_1}\,(z_1 + 1)^{-\beta_2}z_2^{-\beta_3} \,d\eta,
\end{align*}
which yields an identity of $A$-hypergeometric integrals in the case when $\beta_2$ is a negative integer
using the same reasoning as in the previous example.
\end{example}

\section{On the absence of integral representations of Apell's $F_4$.}
\label{sec:AppellF4}

The standard form of an integral representation of classical hypergeometric functions
is as a dehomogenized Euler type integral~\eqref{eqn:IntegralRepresentation} over the positive
orthant $\RR_+^{d-1}$, with a coefficient which is a quotient of gamma functions in the parameters
$\beta$. Such expressions are known, e.g., for Gauss' hypergeometric function and 
Lauricella $F_D$.
However, such an expression is not known in the case of Apell's hypergeometric
function $F_4$, a special case of Lauricella $F_C$. In this section we
prove the following theorem, as an application of the results of \S
\ref{sec:Symmetries}. 

\begin{theorem}
\label{thm:AppellF4}
The Apell hypergeometric function $F_4$ does not admit any integral representation in the form of a dehomogenized
Euler type integral taken over a cycle $\sigma$ which is a rotation of the positive orthant.
\end{theorem}

\begin{proof}
Apell's hypergeometric function $F_4$ can be realized as an 
$A$-hypergeometric function using the setup from
Example~\ref{ex:LauricellaFC}. More precisely, let $A$ be as in the case $m= 2$ of that example.
In its dehomogenized version, Apell's $F_4$ admits the series representation
\[
F_4(a,b;c,c'; y_1,y_2) = \sum_{r,s = 0}^{\infty} \frac{(a)_{r+s}(b)_{r+s}}{(c)_r(c')_s \,r!\, s!}\, y_1^r\, y_2^s
\]
Then, for $\kappa$ the transpose of $[-a, c-1, c'-1,-b,0,0]$, the function
\[
\frac{x_2^{c-1}x_3^{c'-1}}{x_1^a x_4^b} 
F_4\left(a,b; c,c'; \frac{x_2x_5}{x_1x_4}, \frac{x_3x_6}{x_1x_4} \right) 
\]
is $A$ hypergeometric of parameter $\beta = A \kappa$. 

Let us note that Apell's $F_4$ satisfies the transformation
\begin{align}
\label{eqn:ApellTransformation}
F_4(a,b; c,c'; y_1,y_2) = 
& \,K_1(\beta)(-y_2)^{-a}F_4\left(a,a-c'+1;c,a-b+1; \frac{y_1}{y_2},\frac{1}{y_2}\right) \\
\nonumber& + K_2(\beta) (-y_2)^{b}F_4\left(b-c'+1, b ;c,b-a+1; \frac{y_1}{y_2},\frac{1}{y_2}\right) 
\end{align}
where $K_1$  and $K_2$ are certain quotients of $\Gamma$ functions in
the parameters $\beta = A\kappa$, see \cite[\S 5.11]{EMOT53}.
We consider the polytope symmetry of $A$ encoded by
\[
T = \left[\begin{smallmatrix}
 \phantom{-}1 & 0 & \phantom{-}0 & \phantom{-}0 \\
 \phantom{-}0 & 1 & \phantom{-}0 & \phantom{-}0 \\
 \phantom{-}0 & 0 & \phantom{-}1 & \phantom{-}0 \\
 -1 & 2 & -1 & -1
\end{smallmatrix}\right] 
\quad \text{and}\quad
P = \left[\begin{smallmatrix}
0 	& 0 	& 1 	& 0 & 0 & 0  \\ 
0	& 1 	& 0 	& 0 & 0 & 0 \\ 
1 	& 0 	& 0 	& 0 & 0 & 0 \\ 
0	& 0 	& 0	& 0 & 0 & 1\\
0 & 0 & 0 & 0 & 1 & 0 \\
0 & 0 & 0 & 1 & 0 & 0
\end{smallmatrix}\right].
\]
Assume that $F_4$ admits an integral representation as a dehomogenized version of 
the Euler type integral \eqref{eqn:IntegralRepresentation} 
taken over some cycle $\sigma$ which is a rotation of the positive orthant.
It then follows from Corollary~\ref{cor:LinearTransformation}, using $T$ and $P$ as above,
that $F_4$ satisfies, for generic parameters, an identity of the form 
\[
F_4(a,b; c,c'; y_1,y_2) = K_3(\beta) (-y_2)^{b}F_4\left(b-c'+1, b ;c,b-a+1; \frac{y_1}{y_2},\frac{1}{y_2}\right).
\]
However, this contradicts the validity of the transformation \eqref{eqn:ApellTransformation}.
Indeed, composing the two transformations, we find that the two functions
appearing in the right hand side of \eqref{eqn:ApellTransformation} are linearly dependent
for generic parameters $\beta$; this is absurd.
\end{proof}

\raggedbottom
\def\cprime{$'$} \def\cprime{$'$}
\providecommand{\MR}{\relax\ifhmode\unskip\space\fi MR }
\providecommand{\MRhref}[2]{%
  \href{http://www.ams.org/mathscinet-getitem?mr=#1}{#2}
}
\providecommand{\href}[2]{#2}
\end{document}